\documentclass[10pt]{amsart}
\usepackage{amsfonts}
\allowdisplaybreaks[4]
\usepackage{epsfig}
\usepackage{color}
\newcommand{\abs}[1]{\left\lvert #1 \right\rvert}

\def\E#1{\mathbb{E}\left \{#1 \right\}}

\def\x{\vk{x}}
\def\njd{\{1 \ldot d\}}

\definecolor{c20}{rgb}{0.,0.7,0.}
\definecolor{c30}{rgb}{0.,0.,1.}
\definecolor{c40}{rgb}{1,0.1,0.7}
\definecolor{c50}{rgb}{1,0,0}
\definecolor{c60}{rgb}{1,0.9,0.1}
\definecolor{c70}{rgb}{0.50,1.00,0.00}

\def\EE#1{{\textcolor{c30}{#1}}}
\def\EE#1{#1}
\def\Eh#1{{\textcolor{c30}{#1}}}
\def\Eh#1{#1}

\def\tE#1{{\textcolor{c30}{#1}}}
\def\tE#1{#1}

\def\cw#1{{\textcolor{c40}{#1}}}
\def\cw#1{#1}

\def\zcw#1{{\textcolor{c40}{#1}}}
\def\zcw#1{#1}

\def\wzc#1{{\textcolor{c40}{#1}}}
\numberwithin{equation}{section}
\newtheorem{theo}{Theorem}[section]

\newtheorem{remark}[theo]{Remark}

\numberwithin{equation}{section}
\newtheorem{theorem}{Theorem}[section]

\newtheorem{lemma}{Lemma}[section]

\newcommand{\pk}[1]{\mathbb{P} \left( #1 \right) }
\newcommand{\EXP}[1]{\exp \left( #1 \right) }

\newcommand{\COM}[1]{}

\def\IF{\infty}

\newcommand{\R}{\mathbb{R}}
\newcommand{\inr}{\in \R}

\newcommand{\BQN}{\begin{eqnarray}}
\newcommand{\EQN}{\end{eqnarray}}
\newcommand{\BQNY}{\begin{eqnarray*}}
\newcommand{\EQNY}{\end{eqnarray*}}
\def\ldot{, \ldots,}

\def\polhk#1{\setbox0=\hbox{#1}{\ooalign{\hidewidth
    \lower1.5ex\hbox{`}\hidewidth\crcr\unhbox0}}}

\newcommand{\kb}[1]{\boldsymbol{#1}}
\newcommand{\vk}[1]{\kb{#1}}

\begin{document}

\title{Maxima of A Triangular Array of Multivariate Gaussian Sequence}%[Maxima of Gaussian Sequence]

\author{Enkelejd  Hashorva}
\address{Enkelejd Hashorva, Department of Actuarial Science, %\\Faculty of Business and Economics\\
University of Lausanne,\\
UNIL-Dorigny, 1015 Lausanne, Switzerland
}
\email{Enkelejd.Hashorva@unil.ch}

\author{Liang Peng}
\address{Liang Peng, School of Mathematics,
Georgia Institute of Technology, Atlanta, GA 30332-0160 \\
}
\email{Peng@math.gatech.edu}

\author{Zhichao Weng}
\address{Zhichao Weng, Department of Actuarial Science, %\\Faculty of Business and Economics\\
University of Lausanne\\
UNIL-Dorigny, 1015 Lausanne, Switzerland
}
\email{zhichao.weng@unil.ch}

\bigskip

\date{\today}
 \maketitle

{\bf Abstract:}  It is known that the normalized maxima of a sequence of independent and identically distributed bivariate normal random vectors with correlation coefficient $\rho \in (-1,1)$ is asymptotically independent, which may seriously underestimate extreme probabilities in practice. By letting $\rho$ depend on the sample
size and go to one with certain rate, H\"usler and Reiss (1989) showed that the normalized maxima can become asymptotically dependent. In this paper, we extend such a study to a triangular array of multivariate Gaussian sequence, which further generalizes the results in Hsing, H\"usler and Reiss (1996) and Hashorva and Weng (2013).



\section{Introduction}

Let $(X_1^{(1)},X_1^{(2)}),\cdots,(X_n^{(1)},X_n^{(2)})$ be independent and identically distributed bivariate normal random vectors with zero means, unit variances and correlation coefficient $\rho\in [-1, 1]$. Put
\begin{equation}\label{un}
u_n(x)=x/a_n+b_n\quad\text{with} \quad a_n=\sqrt{2\ln n}\quad\text{and}\quad b_n=\sqrt{2\ln n}-\frac{\ln\ln n+\ln(4\pi)}{2\sqrt{2\ln n}}.
\end{equation}
When $|\rho|<1$, it is known that for any $x,y\inr$
\[
\Psi_{\rho}(u_n(x), u_n(y)):=\pk{\max_{1\le i\le n}X_i^{(1)}\le u_n(x), \max_{1\le i\le n}X_i^{(2)}\le u_n(y)}\to
e^{-e^{-x}-e^{-y}} \quad\text{as}\quad n\to \IF,
\]
where the limit becomes the joint distribution of two independent Gumbel random variables. In this case, $X_1^{(1)}$ and $X_1^{(2)}$ are called asymptotically independent.
Although normal distributions have many good properties and receive much attention in risk management (see McNeil, Frey and Embrechts (2005) for some overviews), this asymptotic independence property does seriously underestimate certain extreme probabilities in practice. To overcome this drawback,
H\"usler and Reiss (1989) proposed to let
$\rho=\rho(n)$ depend on the sample size $n$ such that
\begin{equation}\label{HR}
(1-\rho(n))\ln n\to\lambda\in [0, \infty] \quad\text{as}\quad n\to\infty,
\end{equation}
and then showed that
\begin{equation}\label{lim}\lim_{n\to\infty}\Psi_{\rho(n)}(u_n(x), u_n(y))%\pk{\max_{1\le i\le n}X_{i}\le u_n(x), \max_{1\le i\le n}Y_i\le u_n(y)}=
=e^{-\Phi(\sqrt{\lambda}+\frac{x-y}{2\sqrt{\lambda}})e^{-y}-\Phi(\sqrt{\lambda}+\frac{y-x}{2\sqrt\lambda})e^{-x}}=:H_\lambda(x,y)
\end{equation}
for $x,y\inr$, where $\Phi$ denotes the standard normal distribution function. It is easy to see that the limit distribution $H_\lambda$ (referred to as the H\"usler-Reiss distribution) is not a product distribution when $\lambda\in (0, \infty)$, i.e., $X_1^{(1)}$ and $X_1^{(2)}$ are asymptotically dependent in this case. Using (\ref{HR}), Frick and Reiss (2013) extended the above limit to the maxima of normal copulas. Some other extensions  of H\"usler and Reiss (1989) to more general triangular
arrays have been made in the literature too as reviewed below.

Consider a triangular array of normal random variables $X_{n,i}, i=1,2,\cdots, n=1,2,\cdots$ such that for each $n$, $\{X_{n,i}, i\ge 1\}$ is a stationary normal sequence with mean zero, variance one and covariance $\rho_{n,j}=\E{X_{n,1}X_{n,j+1}}$.
Motivated by condition (\ref{HR}), by assuming that
\begin{equation}
\label{con1}
(1-\rho_{n,j})\ln n\to\delta_j\in (0, \infty]\quad\text{for all}\quad j\ge 1
\end{equation}
as $n\to\infty$,
 and some other conditions on $\rho_{n,j}$, Hsing, H\"usler and Reiss (1996) showed that
\begin{equation}
\label{res1}
\lim_{n\to\infty}\pk{\max_{1\le j\le n}X_{n,j}\le u_n(x)}=e^{-\theta e^{-x}}
\end{equation}
holds for all $x\inr$, where
\[\theta=\pk{A/2+\sqrt{\delta_k}W_k\le\delta_k\quad\text{for all}\quad k\ge 1\quad\text{such that}\quad \delta_k<\infty}\Eh{,}
\]
with $A$ being a standard exponential random variable independent of $W_k$ and $\{W_k: \delta_k<\infty, k\ge 1\}$ being jointly normal with zero means and
\[ \E{W_iW_j}=\frac{\delta_i+\delta_j-\delta_{|i-j|}}{2\sqrt{\delta_i\delta_j}}.\]
Here $\theta$ is set to be $1$ if all $\delta_j$'s are infinite.
Recently French and Davis (2013) generalized this study to a Gaussian random field on a lattice.

Another extension of H\"usler and Reiss (1989) made by Hashorva and Weng (2013) is to study a triangular array of 2-dimensional stationary Gaussian sequence as follows.

Consider a triangular array of bivariate normal random vectors $X_{n,j}=(X_{n,j}^{(1)},X_{n,j}^{(2)}), j=1,2,\cdots,n=1,2,\cdots$ such that for each $n$, $\{X_{n,j}, j\ge 1\}$ is a Gaussian sequence with mean zero, variance one and covariance
\[\E{X_{n,k}^{(i)}X_{n,l}^{(j)}}=\rho_{ij}(|k-l|,n)\quad\text{for}\quad i,j=1,2.\]
By assuming that
\begin{equation}
\label{con2}
\lim_{n\to\infty}(1-\rho_{12}(0,n))\ln n=\lambda\in [0, \infty]
\end{equation}
and
\begin{equation}
\label{con3}
\sigma:=\max_{1\le k<n, 1\le i,j\le 2}|\rho_{ij}(k,n)|<1,\quad \lim_{n\to\infty}\max_{l_n\le k<n,1\le i,j\le 2}\rho_{ij}(k,n)\ln n=0,
\end{equation}
where $l_n=[n^{\alpha}]$ for some $\alpha\in (0, \frac{1-\sigma}{1+\sigma})$, Hashorva and Weng (2003) proved that
\begin{equation}
\label{res2}
\begin{array}{ll}
\lim_{n\to\infty}\pk{\max_{1\le k\le n}X_{n,k}^{(1)}\le u_n(x), \max_{1\le k\le n}X_{n,k}^{(2)}\le u_n(y)}=H_\lambda(x,y)%\]%\exp\{-\Phi(\sqrt\lambda+\frac{x-y}{2\sqrt\lambda})\exp(-y)-\Phi(\sqrt\lambda+\frac{y-x}{2\sqrt\lambda})\exp(-x)\}
\end{array}
\end{equation}
for all $x,y\inr$. %, where $\Phi(x)$ denotes the distribution function of a standard normal variable.
Taking $y=\infty$ in (\ref{res2}), we have
\[\lim_{n\to\infty}\pk{\max_{1\le k\le n}X_{n,k}^{(1)}\le u_n(x)}=e^{-e^{-x}}\quad\text{for}\quad x\in \R,\]
which may contradict (\ref{res1}).
Note that when (\ref{con1}) holds for $X_{n,k}^{(1)}$ or/and $X_{n,k}^{(2)}$, we have $\lim_{n\to\infty}\sigma=1$ and $S_{n1}$ does not converge to zero (see the bottom of page 323 in Hashorva and Weng (2013)). That is, convergence in (\ref{res2}) excludes the possibility that (\ref{con1}) holds for $X_{n,k}^{(1)}$ and $X_{n,k}^{(2)}$. This motivates us to investigate the limit of $\pk{\max_{1\le k\le n}X_{n,k}^{(1)}\le u_n(x), \max_{1\le k\le n}X_{n,k}^{(2)}\le u_n(y)}$ when (\ref{con2}) holds and (\ref{con1}) holds for both $X_{n,k}^{(1)}$ and $X_{n,k}^{(2)}$. Such a study will generalize the results in both Hsing, H\"usler and Reiss (1996) and Hashorva and Weng (2013).

Some other recent extensions of H\"usler and Reiss (1989) is to drop the Gaussian assumption. For example, Hashorva (2013) studied the maxima of some spherical processes; 
Hashova, Kabluchko and W\"ubker (2012) investigated the maxima of $\chi^2$-random vectors; Manjunath, Frick and Reiss (2012) discussed
 the maxima in the setup of extremal discriminant analysis; Engelke, Kabluchko and Schlather (2014) analyzed the maxima for some type of conditional Gaussian models.

We organize  this paper as follows. Section 2 derives the limit for the normalized componentwise maxima of a triangular array of $d$-dimensional normal random vectors when (\ref{con1}) holds for both marginals and dependence.  All proofs are put in Section 3.

 \section{Main results}
Throughout we consider a triangular array
 $\vk{X}_{n,k}=(X_{n,k}^{(1)},\cdots,X_{n,k}^{(d)}), k=1,2,\cdots,n=1,2,\cdots$ such that for each $n$, $\{\vk{X}_{n,k}, k\ge 1\}$ is a $d$-dimensional  stationary Gaussian sequence
 with mean zero, variance   one and correlations given by
$\E{X_{n,k}^{(i)}X_{n,l}^{(j)}}=\rho_{ij}(|k-l|,n)$
for $k,l=1,2,\cdots$ and $i,j=1,2,\cdots,d$.

Hereafter  $A$ stands for a unit exponential random variable being independent of all other random elements involved and $\vk{x}=(x_1 \ldot x_d)\inr^d$.

\begin{theorem}\label{ath1}
Let $\{\vk{X}_{n,k},k,n \ge1\}$ be a $d$-dimensional stationary Gaussian triangular array satisfying
\begin{equation}\label{eq1.2}
\left\{\begin{array}{ll}
&\lim_{n\to \IF}(1-\rho_{ij}(k,n))\ln n = \delta_{ij}(k)  \in (0,\IF]\quad\text{for}\quad i,j=1,\cdots,d; k=1,2,\cdots\\
&\lim_{n\to\infty}(1-\rho_{ij}(0,n))\ln n=\delta_{ij}(0)\in [0, \infty]\quad\text{for}\quad i, j=1,\cdots,d.
\end{array}\right.
\end{equation}
Suppose that there exist positive integers $l_n,r_n$ satisfying
\begin{equation}\label{eq2.1}
\lim_{n\to \IF}\frac{l_n}{r_n}=0, \quad \lim_{n\to \IF}\frac{r_n}{n}=0,
\end{equation}
\begin{equation}\label{eq2.2}
\lim_{n\to \IF}\frac{n^2}{r_n}\sum_{i,j=1}^d\sum_{s=l_n}^n |\rho_{ij}(s,n)|%(1-\rho_{ij}^2(s,n))^{-\frac{1}{2}}
\EXP{-\frac{2\ln n -\ln\ln n}{1+|\rho_{ij}(s,n)|}}=0
\end{equation}
and
\begin{equation}\label{eq2.3}
\lim_{m\to \IF}\limsup_{n\to \IF}\sum_{i,j=1}^d\sum_{s=m}^{{r_n}} n^{-\frac{1-\rho_{ij}(s,n)}{1+\rho_{ij}(s,n)}}
\frac{(\ln n)^{-\rho_{ij}(s,n)/(1+\rho_{ij}(s,n))}}{\sqrt{1-\rho_{ij}^2(s,n)}}=0.
\end{equation}
Then
\BQN\label{eq:Main}
\lim_{n\to \IF}\pk{\max_{1\le k\le n}X_{n,k}^{(1)}\le u_n(x_1),\cdots,\max_{1\le k\le n}X_{n,k}^{(d)}\le u_n(x_d)}=\EXP{-\sum_{i=1}^d\vartheta_i(\vk{x})\zcw{e^{-x_i}}}, \quad \Eh{\forall \x\inr^d},
\EQN
 where
   \begin{equation}
\label{aeq1}
\begin{array}{ll}
\vartheta_1(\vk{x})=&\mathbb{P}\Bigl (
 \frac{A}{2}+\sqrt{\delta_{t1}(k-1)}W_{k,1}^{(t)}\le \delta_{t1}(k-1)+\frac{x_t-x_1}{2},1\le t\le d,\\
&\quad\mbox{for all}
\quad  k\ge 2 \ \ \mbox{such that} \ \ \delta_{t1}(k-1)<\IF\Bigr)\\
\end{array}
\end{equation}
 and for $i=2,\cdots,d$
\begin{equation}
\label{aeq6}
\begin{array}{ll}
\vartheta_i(\vk{x})=&\mathbb{P}\Bigl (\frac{A}{2}+\sqrt{{\delta}_{si}(0)}W_{1,i}^{(s)}\le {\delta}_{si}(0)+\frac{x_s-x_i}{2}, 1\le s <i, {\delta}_{si}(0)<\infty,\\
&\quad \frac{A}{2}+\sqrt{\delta_{ti}(k-1)}W_{k,i}^{(t)}\le \delta_{ti}(k-1)+\frac{x_t-x_i}{2},1\le t\le d,\\
&\quad\mbox{for all}
\quad  k\ge 2 \ \ \mbox{such that} \ \ \delta_{ti}(k-1)<\IF\Bigr),\\
\end{array}
\end{equation}
where 
%$\{(W_1^{(s)},1\le s<i, W_k^{(t)}, 1\le t\le d),{\delta}_{si}(0)<\IF, \delta_{ti}(k-1)<\IF, k \ge 2\}$
$\{W_{k,i}^{(t)}, 1\le t\le d, \delta_{ti}(k-1)<\IF, k \ge 1\}$ %{\color{red}(Why $i$ here? $1\le i\le d$?)}
are jointly normal with zero means and for each $i=1,\cdots,d$
\BQN \label{eqCOV}
%Cov(W_1^{(s_1)}, W_1^{(s_2)})&=&\frac{-{\delta}_{s_1s_2}(0)+{\delta}_{s_1i}(0)+{\delta}_{s_2i}(0)}
%{2\sqrt{{\delta}_{s_1i}(0){\delta}_{s_2i}(0)}},\\
%Cov(W_1^{(s)}, W_k^{(j)})&=&\frac{-\delta_{sj}(k-1)+\delta_{ji}(k-1)+{\delta}_{si}(0)}{2\sqrt{{\delta}_{si}(0)\delta_{ji}(k-1)}},\\
%Cov(W_k^{(j)}, W_k^{(t)})&=&\frac{\delta_{ji}(k-1)+\delta_{ti}(k-1)-{\delta}_{jt}(0)}{2\sqrt{\delta_{ji}(k-1)\delta_{ti}(k-1)}},\\
 Cov(W_{k,i}^{(j)}, W_{l,i}^{(t)})&=&\frac{\delta_{ji}(k-1)+\delta_{ti}(l-1)-\delta_{jt}(|k-l|)}{2\sqrt{\delta_{ji}(k-1)\delta_{ti}(l-1)}},
\quad j,t =1,\cdots,d, \quad k,l \ge 1.
 \EQN
%{\color{red}(Why $i$ appears in the left hand of (2.8)? Right hand seems nothing to do with $i$)}
%with
%$j,t =1,\cdots,d$ %, $1\le s_1,s_2 <i$
% and $k,l \ge 1$.
\end{theorem}

\begin{remark} i) The $\vartheta$'s above should be set to 1 if all $\delta$'s involved are equal to infinity.
Clearly, if only a finite number of $\delta$'s is not equal to infinity, then $\vartheta$'s are all positive and
thus the limit in \eqref{eq:Main} is a  max-stable distribution function.
As mentioned in Remark 2 of French and Davis (2013), for some tractable correlation functions it is possible to show
that $\vartheta$'s are positive.\\
%, in general we cannot guaranty that $\vartheta$'s are positive. \\
ii) Note that $\vartheta_i(\x)$ does not depend on $x_i$ for each $i\le d$. In the particular case that $\vartheta_i(\x)$ is a non-degenerate distribution function, then clearly $G(\x)= e^{- \sum_{i=1}^d \vartheta_i(\x) e^{-x_i}}$ is a max-stable $d$-dimensional distribution function.\\ If condition (2.1) holds with $\delta_{ij}(k)=\IF$ for any  index $i,j\le d$ and $k\ge 1$, then clearly
$$ \vartheta_i(\x)= \vartheta_i(x_1 \ldot x_i), \quad \x\inr^d, i \le d.$$
Moreover for this case the limiting distribution $G(\x)= e^{- \sum_{i=1}^d \vartheta_i(x_1 \ldot x_i) e^{-x_i}}$
coincides with the $d$-dimensional max-stable H\"usler-Reiss distribution.
\\
iii)
As in Theorem 2.2 of Hsing, H\"usler and Reiss (1996), conditions (\ref{eq2.1}), (\ref{eq2.2}) and (\ref{eq2.3}) can be replaced by
\[\lim_{n\to\infty}\sum_{1\le i,j\le d}\max_{l_n\le k\le n}|\rho_{ij}(k,n)|\ln n=0 \quad\text{for some}\quad l_n=o(n)\]
and
\[\lim_{m\to \IF}\limsup_{n\to \IF}\sum_{i,j=1}^d\sum_{s=m}^{{l_n}} n^{-\frac{1-\rho_{ij}(s,n)}{1+\rho_{ij}(s,n)}}
\frac{(\ln n)^{-\rho_{ij}(s,n)/(1+\rho_{ij}(s,n))}}{\sqrt{1-\rho_{ij}^2(s,n)}}=0.
\] These last two conditions are easier to check than those in Theorem \ref{ath1}.\\

\end{remark}

\section{Proofs}

For notational simplicity we shall define
$$M_{k,l}^{(i)}=\max_{k<s\le l}X_{n,s}^{(i)}, \quad M_{l}^{(i)}=M_{0,l}^{(i)}=\max_{1\le s \le l}X_{n,l}^{(i)},\quad M_{l,l}^{(i)}=-\IF$$
for $i=1,2,\cdots,d$, $k=1,\cdots,l$ and $l=1,\cdots,n$.
Before proving the theorem, we need some lemmas.

\begin{lemma}\label{le1}
For any $n\times d$ random matrix $\{X_{n,k}^{(i)}, 1 \le k \le n, 1 \le i \le d\}$ and any vector of constants $(u^{(1)}, \cdots, u^{(d)})$ we have
\begin{eqnarray}\label{addeq6}
\pk{\bigcup_{i=1}^d \{M_{n}^{(i)}>u^{(i)}\}}
&=&\sum_{k=1}^n\pk{X_{n,k}^{(1)}>u^{(1)},
\bigcap_{t=1}^d\{M_{k,n}^{(t)}\le u^{(t)}\}}\\
&&+\sum_{i=2}^d\sum_{k=1}^n\pk{X_{n,k}^{(i)}>u^{(i)},
\bigcap_{s=1}^{i-1}\{M_{k-1,n}^{(s)}\le u^{(s)}\},\bigcap_{t=i}^d\{M_{k,n}^{(t)}\le u^{(t)}\}}.\nonumber
%=\sum_{i=1}^d\sum_{k=1}^n\pk{X_{n,k}^{(i)}>u^{(i)},
%\bigcap_{s=1}^{i-1}\{M_{k-1,n}^{(s)}\le u^{(s)}\},\bigcap_{t=i}^d\{M_{k,n}^{(t)}\le u^{(t)}\}}.
\end{eqnarray}
\end{lemma}

\begin{proof}
The case of $d=1$ directly follows from O'Brien (1987). We shall prove the case of $d=2$ and then use the induction method to conclude that the lemma holds for any $d\ge 2$.

%First, we will iteratively intersect the event $\{M_{s,n}^{(i)}>u^{(i)}\}$, $0\le s<n$, $1\le i\le d$ with the event
%$\{X_{n,k}^{(i)}>u^{(i)}\}\cup\{X_{n,k}^{(i)}\le u^{(i)}\}$ for changing $k$, sum over the disjoint
%event, and then simplify. Starting with $k=n$, we have
%$$\pk{M_{s,n}^{(i)}>u^{(i)}}=\pk{X_{n,n}^{(i)}>u^{(i)}}+\pk{M_{s,n-1}^{(i)}>u^{(i)},X_{n,n}^{(i)}\le %u^{(i)}}.$$
%Continuing for $k=n-1$ and until $k=s+1$, we get
It is straightforward to check that for any $s\ge 0$ and $i=1,\cdots,d,$
\BQNY
\pk{M_{s,n}^{(i)}>u^{(i)}}&=&\pk{X_{n,n}^{(i)}>u^{(i)}}+\pk{M_{s,n-1}^{(i)}>u^{(i)}, X_{n,n}^{(i)}\le u^{(i)}}\\
&=&\pk{X_{n,n}^{(i)}>u^{(i)}, M_{n,n}^{(i)}\le u^{(i)}}+\pk{X_{n,n-1}^{(i)}>u^{(i)}, X_{n,n}^{(i)}\le u^{(i)}}\\
&&+\pk{M_{s,n-2}^{(i)}>u^{(i)},X_{n,n-1}^{(i)}\le u^{(i)}, X_{n,n}^{(i)}\le u^{(i)}}\\
&=&\pk{X_{n,n}^{(i)}>u^{(i)}, M_{n,n}^{(i)}\le u^{(i)}}+\pk{X_{n,n-1}^{(i)}>u^{(i)}, M_{n-1,n}^{(i)}\le u^{(i)}}\\
&&+\pk{M_{s,n-2}^{(i)}>u^{(i)},X_{n,n-1}^{(i)}\le u^{(i)}, X_{n,n}^{(i)}\le u^{(i)}}.
\EQNY
Continuing  the above decomposition, we have
\begin{equation}\label{addeq7}
\pk{M_{s,n}^{(i)}>u^{(i)}}=\sum_{k=s+1}^n\pk{X_{n,k}^{(i)}>u^{(i)},
M_{k,n}^{(i)}\le u^{(i)}}
\end{equation}
for any $s\ge 0$. For proving that (\ref{addeq6}) holds for the case of $d=2$, we first note that
\BQN\label{aeq2}
%\begin{array}{ll}
\lefteqn{\pk{M_{n}^{(1)}\le u^{(1)},M_{n}^{(2)}> u^{(2)}}}\notag\\
&=&\sum_{k=1}^{n}\pk{X_{n,k}^{(2)}>u^{(2)},M_{k,n}^{(2)}\le u^{(2)},M_{n}^{(1)}\le u^{(1)}}\notag\\
&=&\sum_{k=1}^{n}\pk{X_{n,k}^{(2)}>u^{(2)},M_{k,n}^{(2)}\le u^{(2)},M_{k-1,n}^{(1)}\le u^{(1)}}\notag\\
&&-\sum_{k=1}^{n}\pk{X_{n,k}^{(2)}>u^{(2)},M_{k,n}^{(2)}\le u^{(2)},M_{k-1,n}^{(1)}\le u^{(1)}, M_{k-1}^{(1)}>u^{(1)}}\notag\\
&=&\sum_{k=1}^{n}\pk{X_{n,k}^{(2)}>u^{(2)},M_{k,n}^{(2)}\le u^{(2)},M_{k-1,n}^{(1)}\le u^{(1)}}\notag\\
&&-\sum_{k=1}^{n}\sum_{l=1}^{k-1}\pk{X_{n,k}^{(2)}>u^{(2)},M_{k,n}^{(2)}\le u^{(2)},X_{n,l}^{(1)}>u^{(1)},M_{l,n}^{(1)}\le u^{(1)}}\notag\\
&=&\sum_{k=1}^{n}\pk{X_{n,k}^{(2)}>u^{(2)},M_{k,n}^{(2)}\le u^{(2)},M_{k-1,n}^{(1)}\le u^{(1)}}\notag\\
&&-\sum_{l=1}^{n-1}\sum_{k=l+1}^{n}\pk{X_{n,k}^{(2)}>u^{(2)},M_{k,n}^{(2)}\le u^{(2)},X_{n,l}^{(1)}>u^{(1)},M_{l,n}^{(1)}\le u^{(1)}}\notag\\
&=&\sum_{k=1}^{n}\pk{X_{n,k}^{(2)}>u^{(2)},M_{k,n}^{(2)}\le u^{(2)},M_{k-1,n}^{(1)}\le u^{(1)}}
%\notag\\
%&&
-\sum_{l=1}^{n-1}\pk{M_{l,n}^{(2)}> u^{(2)},X_{n,l}^{(1)}>u^{(1)},M_{l,n}^{(1)}\le u^{(1)}},
%\end{array}\end{equation}
\EQN
which can be used to show that
\BQNY
\lefteqn{\pk{\{M_{n}^{(1)}>u^{(1)}\}\cup \{M_{{n}}^{(2)}> u^{(2)}\}}}\\
&=&\pk{M_{n}^{(1)}> u^{(1)}}+\pk{M_{n}^{(1)}\le u^{(1)},M_{{n}}^{(2)}> u^{(2)}}\\
&=&\sum_{l=1}^{n}\pk{X_{n,l}^{(1)}>u^{(1)},M_{l,n}^{(1)}\le u^{(1)}}+\sum_{k=1}^{n}\pk{X_{n,k}^{(2)}>u^{(2)},M_{k,n}^{(2)}\le u^{(2)},M_{k-1,n}^{(1)}\le u^{(1)}}\\
&&-\sum_{l=1}^{n-1}\pk{M_{l,n}^{(2)}> u^{(2)},X_{n,l}^{(1)}>u^{(1)},M_{l,n}^{(1)}\le u^{(1)}}\\
&=&\pk{X_{n,n}^{(1)}>u^{(1)}}+\left(\sum_{l=1}^{n-1}\pk{X_{n,l}^{(1)}>u^{(1)},M_{l,n}^{(1)}\le u^{(1)}}-\sum_{l=1}^{n-1}\pk{M_{l,n}^{(2)}> u^{(2)},X_{n,l}^{(1)}>u^{(1)},M_{l,n}^{(1)}\le u^{(1)}}\right)\\
&&+\sum_{k=1}^{n}\pk{X_{n,k}^{(2)}>u^{(2)},M_{k,n}^{(2)}\le u^{(2)},M_{k-1,n}^{(1)}\le u^{(1)}}\\
&=&\pk{X_{n,n}^{(1)}>u^{(1)}}
+\sum_{l=1}^{n-1}\pk{X_{n,l}^{(1)}>u^{(1)},M_{l,n}^{(1)}\le u^{(1)},M_{l,n}^{(2)}\le u^{(2)}}\\
&&+\sum_{k=1}^{n}\pk{X_{n,k}^{(2)}>u^{(2)},M_{k,n}^{(2)}\le u^{(2)},M_{k-1,n}^{(1)}\le u^{(1)}}\\
&=&\sum_{l=1}^{n}\pk{X_{n,l}^{(1)}>u^{(1)},M_{l,n}^{(1)}\le u^{(1)},M_{l,n}^{(2)}\le u^{(2)}}
+\sum_{k=1}^{n}\pk{X_{n,k}^{(2)}>u^{(2)},M_{k,n}^{(2)}\le u^{(2)},M_{k-1,n}^{(1)}\le u^{(1)}},
\EQNY %\end{array}\]
i.e., \eqref{addeq6} holds for $d=2$.

Next, suppose that (\ref{addeq6}) holds for $d=m-1>2$, i.e.,
\BQN\label{eq1}
\pk{\bigcup_{i=1}^{m-1} \{M_{n}^{(i)}>u^{(i)}\}}
&=&\sum_{k=1}^n\pk{X_{n,k}^{(1)}>u^{(1)},\bigcap_{t=1}^{m-1}\{M_{k,n}^{(t)}\le u^{(t)}\}}\notag \\
&&+\sum_{i=2}^{m-1}\sum_{k=1}^n\pk{X_{n,k}^{(i)}>u^{(i)},\bigcap_{s=1}^{i-1}\{M_{k-1,n}^{(s)}\le u^{(s)}\},\bigcap_{t=i}^{m-1}\{M_{k,n}^{(t)}\le u^{(t)}\}}.
\EQN
In view of \eqref{addeq7} and \eqref{eq1}, we have
\BQN\label{eq2}
%\begin{array}{ll}
\lefteqn{\pk{\bigcap_{i=1}^{m-1} \{M_n^{(i)}\le u^{(i)}\}, M_n^{(m)}>u^{(m)}}}\notag\\
&=&\sum_{k=1}^n\pk{X_{n,k}^{(m)}>u^{(m)},M_{k,n}^{(m)}\le u^{(m)},\bigcap_{i=1}^{m-1} \{M_n^{(i)}\le u^{(i)}\}}\notag\\
&=&\sum_{k=1}^n\pk{X_{n,k}^{(m)}>u^{(m)},M_{k,n}^{(m)}\le u^{(m)},\bigcap_{i=1}^{m-1} \{M_{k-1,n}^{(i)}\le u^{(i)}\}}\notag\\
&&-\sum_{k=1}^n\pk{X_{n,k}^{(m)}>u^{(m)},M_{k,n}^{(m)}\le u^{(m)},\bigcap_{i=1}^{m-1} \{M_{k-1,n}^{(i)}\le u^{(i)}\},\bigcup_{j=1}^{m-1} \{M_{k-1}^{(j)}>u^{(j)}\}}\notag\\
&=&\sum_{k=1}^n\pk{X_{n,k}^{(m)}>u^{(m)},M_{k,n}^{(m)}\le u^{(m)},\bigcap_{i=1}^{m-1} \{M_{k-1,n}^{(i)}\le u^{(i)}\}}\notag\\
&&-\sum_{k=1}^n\sum_{l=1}^{k-1}\pk{X_{n,k}^{(m)}>u^{(m)},M_{k,n}^{(m)}\le u^{(m)},\bigcap_{i=1}^{m-1} \{M_{k-1,n}^{(i)}\le u^{(i)}\},X_{n,l}^{(1)}>u^{(1)},\bigcap_{t=1}^{m-1} \{M_{l,k-1}^{(t)}\le u^{(t)}\}}\notag\\
&&-\sum_{k=1}^n\sum_{j=2}^{m-1}\sum_{l=1}^{k-1}\mathbb{P}\left(X_{n,k}^{(m)}>u^{(m)},M_{k,n}^{(m)}\le u^{(m)},
\bigcap_{i=1}^{m-1} \{M_{k-1,n}^{(i)}\le u^{(i)}\},\right. \notag\\
&&\qquad \left. X_{n,l}^{(j)}>u^{(j)},\bigcap_{s=1}^{j-1}\{M_{l-1,k-1}^{(s)}\le u^{(s)}\},\bigcap_{t=j}^{m-1}\{M_{l,k-1}^{(t)}\le u^{(t)}\}\right)\notag\\
&=&\sum_{k=1}^n\pk{X_{n,k}^{(m)}>u^{(m)},M_{k,n}^{(m)}\le u^{(m)},\bigcap_{i=1}^{m-1} \{M_{k-1,n}^{(i)}\le u^{(i)}\}}\notag\\
&&-\sum_{l=1}^{n-1}\sum_{k=l+1}^{n}\pk{X_{n,k}^{(m)}>u^{(m)},M_{k,n}^{(m)}\le u^{(m)},\bigcap_{i=1}^{m-1} \{M_{l,n}^{(i)}\le u^{(i)}\},X_{n,l}^{(1)}>u^{(1)}}\notag\\
&&-\sum_{j=2}^{m-1}\sum_{l=1}^{n-1}\sum_{k=l+1}^{n}
\pk{X_{n,k}^{(m)}>u^{(m)},M_{k,n}^{(m)}\le u^{(m)},
X_{n,l}^{(j)}>u^{(j)},\bigcap_{s=1}^{j-1}\{M_{l-1,n}^{(s)}\le u^{(s)}\},\bigcap_{t=j}^{m-1}\{M_{l,n}^{(t)}\le u^{(t)}\}}\notag\\
&=&\sum_{k=1}^n\pk{X_{n,k}^{(m)}>u^{(m)},M_{k,n}^{(m)}\le u^{(m)},\bigcap_{i=1}^{m-1} \{M_{k-1,n}^{(i)}\le u^{(i)}\}}\notag\\
&&-\sum_{l=1}^{n-1}\pk{M_{l,n}^{(m)}> u^{(m)},\bigcap_{i=1}^{m-1} \{M_{l,n}^{(i)}\le u^{(i)}\},X_{n,l}^{(1)}>u^{(1)}}\notag\\
&&-\sum_{j=2}^{m-1}\sum_{l=1}^{n-1}\pk{
X_{n,l}^{(j)}>u^{(j)},\bigcap_{s=1}^{j-1}\{M_{l-1,n}^{(s)}\le u^{(s)}\},\bigcap_{t=j}^{m-1}\{M_{l,n}^{(t)}\le u^{(t)}\}, M_{l,n}^{(m)}>u^{(m)}}.
%\end{array}
\EQN%\end{equation}
It follows from \eqref{eq1} and (\ref{eq2}) that
\begin{eqnarray*}
&&\pk{\bigcup_{i=1}^{m} \{M_{n}^{(i)}>u^{(i)}\}}\\
&=&\pk{\bigcup_{i=1}^{m-1} \{M_{n}^{(i)}>u^{(i)}\}}+\pk{\bigcap_{i=1}^{m-1} \{M_n^{(i)}\le u^{(i)}\}, M_n^{(m)}>u^{(m)}}\\
&=&\sum_{k=1}^n\pk{X_{n,k}^{(1)}>u^{(1)},\bigcap_{t=1}^{m-1}\{M_{k,n}^{(t)}\le u^{(t)}\}}\\
&&+\sum_{i=2}^{m-1}\sum_{k=1}^n\pk{X_{n,k}^{(i)}>u^{(i)},\bigcap_{s=1}^{i-1}\{M_{k-1,n}^{(s)}\le u^{(s)}\},\bigcap_{t=i}^{m-1}\{M_{k,n}^{(t)}\le u^{(t)}\}}\\
&&+\sum_{k=1}^n\pk{X_{n,k}^{(m)}>u^{(m)},M_{k,n}^{(m)}\le u^{(m)},\bigcap_{i=1}^{m-1} \{M_{k-1,n}^{(i)}\le u^{(i)}\}}\\
&&-\sum_{l=1}^{n-1}\pk{M_{l,n}^{(m)}> u^{(m)},\bigcap_{i=1}^{m-1} \{M_{l,n}^{(i)}\le u^{(i)}\},X_{n,l}^{(1)}>u^{(1)}}\\
&&-\sum_{j=2}^{m-1}\sum_{l=1}^{n-1}\pk{
X_{n,l}^{(j)}>u^{(j)},\bigcap_{s=1}^{j-1}\{M_{l-1,n}^{(s)}\le u^{(s)}\},\bigcap_{t=j}^{m-1}\{M_{l,n}^{(t)}\le u^{(t)}\}, M_{l,n}^{(m)}>u^{(m)}}\\
&=&\left(\sum_{k=1}^n\pk{X_{n,k}^{(1)}>u^{(1)},\bigcap_{t=1}^{m-1}\{M_{k,n}^{(t)}\le u^{(t)}\}}
-\sum_{l=1}^{n-1}\pk{M_{l,n}^{(m)}> u^{(m)},\bigcap_{i=1}^{m-1} \{M_{l,n}^{(i)}\le u^{(i)}\},X_{n,l}^{(1)}>u^{(1)}}\right)\\
&&+\left(\sum_{i=2}^{m-1}\sum_{k=1}^n\pk{X_{n,k}^{(i)}>u^{(i)},\bigcap_{s=1}^{i-1}\{M_{k-1,n}^{(s)}\le u^{(s)}\},\bigcap_{t=i}^{m-1}\{M_{k,n}^{(t)}\le u^{(t)}\}}\right.\\
&&\left.-\sum_{j=2}^{m-1}\sum_{l=1}^{n-1}\pk{
X_{n,l}^{(j)}>u^{(j)},\bigcap_{s=1}^{j-1}\{M_{l-1,n}^{(s)}\le u^{(s)}\},\bigcap_{t=j}^{m-1}\{M_{l,n}^{(t)}\le u^{(t)}\}, M_{l,n}^{(m)}>u^{(m)}}\right)\\
&&+\sum_{k=1}^n\pk{X_{n,k}^{(m)}>u^{(m)},M_{k,n}^{(m)}\le u^{(m)},\bigcap_{i=1}^{m-1} \{M_{k-1,n}^{(i)}\le u^{(i)}\}}\\
%&=&\left(\pk{X_{n,n}^{(1)}>u^{(1)}}
%+\sum_{k=1}^{n-1}\pk{X_{n,k}^{(1)}>u^{(1)},\bigcap_{t=1}^{m-1}\{M_{k,n}^{(t)}\le u^{(t)}\}}\right.\\
%&&\left.-\sum_{l=1}^{n-1}\pk{M_{l,n}^{(m)}> u^{(m)},\bigcap_{i=1}^{m-1} \{M_{l,n}^{(i)}\le u^{(i)}\},X_{n,l}^{(1)}>u^{(1)}}\right)\\
%&&+\left(\sum_{i=2}^{m-1}\pk{X_{n,n}^{(i)}>u^{(i)},\bigcap_{s=1}^{i-1}\{X_{n,n}^{(s)}\le u^{(s)}\}}\right.\\
%&&+\sum_{i=2}^{m-1}\sum_{k=1}^{n-1}\pk{X_{n,k}^{(i)}>u^{(i)},\bigcap_{s=1}^{i-1}\{M_{k-1,n}^{(s)}\le u^{(s)}\},\bigcap_{t=i}^{m-1}\{M_{k,n}^{(t)}\le %u^{(t)}\}}\\
%&&\left.-\sum_{j=2}^{m-1}\sum_{l=1}^{n-1}\pk{
%X_{n,l}^{(j)}>u^{(j)},\bigcap_{s=1}^{j-1}\{M_{l-1,n}^{(s)}\le u^{(s)}\},\bigcap_{t=j}^{m-1}\{M_{l,n}^{(t)}\le u^{(t)}\}, %M_{l,n}^{(m)}>u^{(m)}}\right)\\
%&&+\sum_{k=1}^n\pk{X_{n,k}^{(m)}>u^{(m)},M_{k,n}^{(m)}\le u^{(m)},\bigcap_{i=1}^{m-1} \{M_{k-1,n}^{(i)}\le u^{(i)}\}}\\
&=&\left(\pk{X_{n,n}^{(1)}>u^{(1)}}
+\sum_{l=1}^{n-1}\pk{X_{n,l}^{(1)}>u^{(1)},\bigcap_{i=1}^{m-1} \{M_{l,n}^{(i)}\le u^{(i)}\},M_{l,n}^{(m)}\le u^{(m)}}\right)\\
&&+\left(\sum_{i=2}^{m-1}\pk{X_{n,n}^{(i)}>u^{(i)},\bigcap_{s=1}^{i-1}\{X_{n,n}^{(s)}\le u^{(s)}\}}\right.\\
&&\left.+\sum_{j=2}^{m-1}\sum_{l=1}^{n-1}\pk{
X_{n,l}^{(j)}>u^{(j)},\bigcap_{s=1}^{j-1}\{M_{l-1,n}^{(s)}\le u^{(s)}\},\bigcap_{t=j}^{m-1}\{M_{l,n}^{(t)}\le u^{(t)}\}, M_{l,n}^{(m)}\le u^{(m)}}\right)\\
&&+\sum_{k=1}^n\pk{X_{n,k}^{(m)}>u^{(m)},M_{k,n}^{(m)}\le u^{(m)},\bigcap_{i=1}^{m-1} \{M_{k-1,n}^{(i)}\le u^{(i)}\}}\\
&=&\sum_{l=1}^{n}\pk{X_{n,l}^{(1)}>u^{(1)},\bigcap_{i=1}^{m} \{M_{l,n}^{(i)}\le u^{(i)}\}}+\sum_{j=2}^{m-1}\sum_{l=1}^{n}\pk{
X_{n,l}^{(j)}>u^{(j)},\bigcap_{s=1}^{j-1}\{M_{l-1,n}^{(s)}\le u^{(s)}\},\bigcap_{t=j}^{m}\{M_{l,n}^{(t)}\le u^{(t)}\}}\\
&&+\sum_{k=1}^n\pk{X_{n,k}^{(m)}>u^{(m)},M_{k,n}^{(m)}\le u^{(m)},\bigcap_{i=1}^{m-1} \{M_{k-1,n}^{(i)}\le u^{(i)}\}}\\
&=&\sum_{l=1}^{n}\pk{X_{n,l}^{(1)}>u^{(1)},\bigcap_{i=1}^{m} \{M_{l,n}^{(i)}\le u^{(i)}\}}\\
&&+\sum_{j=2}^{m}\sum_{l=1}^{n}\pk{
X_{n,l}^{(j)}>u^{(j)},\bigcap_{s=1}^{j-1}\{M_{l-1,n}^{(s)}\le u^{(s)}\},\bigcap_{t=j}^{m}\{M_{l,n}^{(t)}\le u^{(t)}\}},
\end{eqnarray*}
i.e., (\ref{addeq6}) holds for $d=m$. Hence the lemma follows from the induction method. \end{proof}

\begin{lemma}\label{ale1}
Let $\{\vk{X}_{n,k},k,n\ge 1\}$ be a $d$-dimensional stationary Gaussian triangular array.
If there exist positive integers $l_n$ and $r_n$ such that \eqref{eq2.1} and \eqref{eq2.2} hold, then we have for any $x_i\inr, i\le d$
\begin{equation}\label{eq3}
\lim_{n\to \IF}\left(\pk{\bigcap_{i=1}^{d} \{M_{n}^{(i)}\le u_n(x_i)\}}-\left(\pk{\bigcap_{i=1}^{d} \{M_{r_n}^{(i)}\le u_n(x_i)\}}\right)^{q_n}\right)=0,
\end{equation}
where $q_n=[n/r_n]$.
\end{lemma}

\begin{proof}  Define $N_n=\{1,2,\cdots,n\}$ for any positive integer $n$ and set
\begin{eqnarray*}
N_{{r_n}q_n}=(I_1\cup J_1)\cup(I_2 \cup J_2)\cup\ldots \cup(I_{q_n} \cup J_{q_n}),
\end{eqnarray*}
with $I_s=\{(s-1){r_n}+1, \ldots,s{r_n}-l_n\}$ and $J_s=\{s{r_n}-l_n+1,\ldots,s{r_n}\}$ for $s=1,2, \ldots, q_n$.
Since ${r_n}q_n\leq n<({r_n}+1)q_n<{r_n}q_n+l_n$, we get $|N_n\backslash N_{{r_n}q_n}|<q_n<l_n$, where $\abs{K}$ means the length of the interval  $K\subset  \R$.
Further, define sets $I_{q_n+1}$ and $J_{q_n+1}$ by
\begin{eqnarray*}
I_{q_n+1} &=& \{{r_n}q_n-{r_n}+l_n+1,\ldots,{r_n}q_n-1,{r_n}q_n\},
\\
J_{q_n+1} &=& \{{r_n}q_n+1,\ldots,{r_n}q_n+l_n-1,{r_n}q_n+l_n\}.
\end{eqnarray*}
Clearly,  $|I_{q_n+1}|={r_n}-l_n$, $|J_{q_n+1}|=l_n$ and
 $I_{q_n+1}\subset N_{r_nq_n}$ and $N_n\backslash N_{r_nq_n}\subset J_{q_n+1}$.
Using the fact that
$$l_n=o(r_n), \quad l_n=o(n), \quad \lim_{n\to \IF} n(1-\Phi(u_n(x_i)))= e^{-x_i}$$
we obtain
\begin{eqnarray*}
0
&\leq&
\pk{\bigcap_{s=1}^{q_n}\bigcap_{i=1}^d\{M^{(i)}(I_s)\le u_n(x_i)\}}-\pk{\bigcap_{i=1}^d\{M^{(i)}_n\le u_n(x_i)\}}\nonumber\\
&\leq&
\sum^{q_n+1}_{s=1}\sum_{i=1}^d\pk{M^{(i)}(I_s)\le u_n(x_i)< M^{(i)}(J_s)}\nonumber\\
&\leq&
\sum^{q_n+1}_{s=1}\sum_{i=1}^d\pk{u_n(x_i)< M^{(i)}(J_s)}\nonumber\\
&\le &
(q_n+1)l_n \sum_{i=1}^d (1-\Phi(u_n(x_i))) \\
&\to& 0
\end{eqnarray*}
as $n\to\infty$, where $M^{(i)}(I_s)=\max_{j\in I_s}X_{n,j}^{(i)}$.
 Using Berman's inequality given in Li and Shao (2001) (see also  Piterbarg (1996))
and \eqref{eq2.2}%, %for all $A,B\subset\{1,\ldots,n\}$ such that $b-a\ge l_n$ for any $a\in A$ and $b\in B$, we have with
%$E_{n,i}(\vk{x}):=\bigcap_{ 1 \le k \le d}\{X_{n,i}^{(k)}\le u_n(x_k)\}$
\begin{eqnarray*}
&&\left|\pk{\bigcap_{s=1}^{q_n}\bigcap_{i=1}^d\{M^{(i)}(I_s)\le u_n(x_i)\}}-
\prod_{s=1}^{q_n}\pk{\bigcap_{i=1}^d\{M^{(i)}(I_s)\le u_n(x_i)\}}\right|\\
&\le& (q_n-1)\frac{1}{2\pi}\sum_{i,j=1}^d\sum_{1\le s<t\le n, t-s>l_n}\Bigl|\arcsin(\rho_{ij}(t-s,n))\Bigr|\EXP{-\frac{u_n^2(x_i)+u_n^2(x_j)}{2(1+|\rho_{ij}(t-s,n)|)}}\\
&\le&C\frac{n^2}{r_n}\sum_{i,j=1}^d\sum_{s=l_n}^n|\rho_{ij}(s,n)|\EXP{-\frac{2\ln n-\ln\ln n}{1+|\rho_{ij}(s,n)|}}\\
&\to& 0\quad\text{as} \quad n\to\IF,
\end{eqnarray*}
where $C$ is some positive constant.
Since further
\begin{eqnarray*}
0
&\leq&
\prod_{s=1}^{q_n}\pk{\bigcap_{i=1}^d\{M^{(i)}(I_s)\le u_n(x_i)\}}
-\prod_{s=1}^{q_n}\pk{\bigcap_{i=1}^d\{M^{(i)}(I_s\cup J_s)\le u_n(x_i)\}}\\
&\leq&
\sum^{q_n}_{s=1}\sum_{i=1}^d\pk{u_n(x_i)< M^{(i)}(J_s)}
\to 0
\end{eqnarray*}
as $n\to \IF$, the lemma follows.  \end{proof}

\begin{remark} \label{remE}  If $\{s_n,n\ge 1\}$ is a sequence of positive integers such that $s_n=o(n)$ and $r_n=o(s_n)$, then clearly
both \eqref{eq2.1} and \eqref{eq2.2} hold with $r_n$ replaced by $s_n$. From the proof above we see that these two conditions
are the only assumptions of Lemma \ref{ale1}. Hence \eqref{eq3} still holds if we substitute $q_n$ by $t_n= [n/s_n]$.
\end{remark}

\begin{lemma}\label{le3.5}
Under the conditions of Theorem \ref{ath1}, for any bounded index set $K \subset \{2,3, \ldots\}$ and  each $c \in \{2, \ldots,d\}$ we have
\begin{equation}\label{aeq7}
\begin{array}{ll}
&\lim_{n\to \IF}\pk{X_{n,1}^{(s)} \le u_n(x_s),1\le s <c,X_{n,k}^{(t)}\le u_n(x_t),1\le t\le d, k\in K | X_{n,1}^{(c)}> u_n(x_c)}\\
=&\mathbb{P}(\frac{A}{2}+\sqrt{{\delta}_{sc}(0)}W_{1,\wzc{c}}^{(s)}\le {\delta}_{sc}(0)+\frac{x_s-x_c}{2}, 1\le s <c, {\delta}_{sc}(0)<\IF,\\
&\quad \frac{A}{2}+\sqrt{\delta_{tc}(k-1)}W_{k,\wzc{c}}^{(t)}\le \delta_{tc}(k-1)+\frac{x_t-x_c}{2},1\le t\le d, \mbox{for all}\ \  k \in K
\ \mbox{such that}\ \ \delta_{tc}(k-1)<\IF).
\end{array}
\end{equation}
Further
\begin{equation}\label{aeq8}
\begin{array}{ll}
&\lim_{n\to \IF}\pk{X_{n,k}^{(t)}\le u_n(x_t),1\le t\le d, k\in K | X_{n,1}^{(1)}> u_n(x_1)}\\
=&\mathbb{P}(\frac{A}{2}+\sqrt{\delta_{t1}(k-1)}W_{k,1}^{(t)}\le \delta_{t1}(k-1)+\frac{x_t-x_1}{2},1\le t\le d, \mbox{for all}\ \  k \in K
\ \  \mbox{such that}\ \ \delta_{t1}(k-1)<\IF),
\end{array}
\end{equation}
and $\{W_{k,i}^{(t)}, 1\le t\le d, \delta_{ti}(k-1)<\IF, k\in \{1\}\cup K\}$
are jointly normal with zero means and
\BQNY
 Cov(W_{k,i}^{(j)}, W_{l,i}^{(t)})&=&\frac{\delta_{ji}(k-1)+\delta_{ti}(l-1)-\delta_{jt}(|k-l|)}{2\sqrt{\delta_{ji}(k-1)\delta_{ti}(l-1)}},
\quad i,j,t =1,\cdots,d, \quad k,l \in \{1\} \cup K.
\EQNY
%where % $E$ denoting a standard exponential random variable independent of
%$\{(W_1^{(s)},1\le s<c, W_k^{(t)}, 1\le t\le d), \delta_{tc}(k-1)<\IF, k \in K\}$
%are  jointly normal with zero means and
%\BQNY
%Cov(W_1^{(s_1)}, W_1^{(s_2)})&=&\frac{-\widetilde{\delta}_{s_1s_2}+\widetilde{\delta}_{s_1c}+\widetilde{\delta}_{s_2c}}
%{2\sqrt{\widetilde{\delta}_{s_1c}\widetilde{\delta}_{s_2c}}},\\
 %Cov(W_1^{(s)}, W_k^{(i)})&=&\frac{-\delta_{is}(k-1)+\delta_{ic}(k-1)+\widetilde{\delta}_{sc}}{2\sqrt{\widetilde{\delta}_{sc}\delta_{ic}(k-1)}},\\
 %Cov(W_k^{(i)}, W_k^{(j)})&=&\frac{-\widetilde{\delta}_{ij}+\delta_{ic}(k-1)+\delta_{jc}(k-1)}{2\sqrt{\delta_{ic}(k-1)\delta_{jc}(k-1)}},\\
% Cov(W_k^{(i)}, W_l^{(j)})&=&\frac{-\delta_{ij}(|k-l|)+\delta_{ic}(k-1)+\delta_{jc}(l-1)}{2\sqrt{\delta_{ic}(k-1)\delta_{jc}(l-1)}},
%\EQNY
%with
%$i,j \in \{1,\ldots,d\}$, $s_1,s_2\in\{1,\ldots,c-1\}$ and $k \neq l \in K$.
\end{lemma}

\begin{proof}
%The proof follows by utlising the idea of the proof of Theorem 5.1 in Berman (1982).
%Due to similarities and for completeness
We follow  the arguments in the proof of Lemma 4.1 in Hsing, H\"usler and Reiss (1996).
First like (4.1) therein we have for each $c \in \{2, \ldots,d\}$,
\begin{eqnarray}\label{eq3.4}
&&\pk{X_{n,1}^{(s)} \le u_n(x_s),1\le s <c,X_{n,k}^{(t)}\le u_n(x_t),1\le t\le d, k\in K | X_{n,1}^{(c)}> u_n(x_c)}\nonumber\\
&\sim&\int_0^{\IF}\pk{X_{n,1}^{(s)} \le u_n(x_s),1\le s <c,X_{n,k}^{(t)}\le u_n(x_t),1\le t\le d, k\in K \big| X_{n,1}^{(c)}= T_n(x_c,z)}\nonumber\\
&&\times\EXP{-z-\frac{z^2}{2u_n^2(x_c)}} dz,
\end{eqnarray}
where $T_n(x_c,z)=u_n(x_c)+z/u_n(x_c)$.
%Let $\{(Y_{n,1}^{(s)},1\le s<c, Y_{n,k}^{(t)},1\le t\le d,k\in K)\}$ have the same distribution  as the conditional distribution of
%$\{(X_{n,1}^{(s)},1\le s<c, X_{n,k}^{(t)},1\le t\le d,k\in K)\}$  given
%$X_{n,1}^{(c)}=T_n(x_c,z)$.  Then
%\begin{eqnarray*}
%\E{Y_{n,1}^{(s)}}={\rho}_{sc}(0,n)T_n(x_c,z),\quad \E{Y_{n,k}^{(t)}}=\rho_{tc}(k-1,n)T_n(x_c,z),
%\end{eqnarray*}
%\begin{eqnarray*}
%Cov( Y_{n,1}^{(s)}, Y_{n,1}^{(s)})= 1-{\rho}_{sc}^2(0,n),
%\quad
%Cov( Y_{n,1}^{(s_1)}, Y_{n,1}^{(s_2)})= {\rho}_{s_1s_2}(0,n)-{\rho}_{s_1c}(0,n){\rho}_{s_2c}(0,n),
%\end{eqnarray*}
%\begin{eqnarray*}
%Cov( Y_{n,k}^{(i)}, Y_{n,k}^{(i)})= 1-\rho_{ic}^2(k-1,n),
%\quad
%Cov( Y_{n,1}^{(s)}, Y_{n,k}^{(i)})= \rho_{is}(k-1,n)-{\rho}_{sc}(0,n)\rho_{ic}(k-1,n),
%\end{eqnarray*}
%\[
%Cov( Y_{n,k}^{(i)}, Y_{n,k}^{(j)})={\rho}_{ij}(0,n)-\rho_{ic}(k-1,n)\rho_{jc}(k-1,n)
%\]
%and
%\[
%Cov( Y_{n,k}^{(i)}, Y_{n,l}^{(j)})= \rho_{ij}(|k-l|,n)-\rho_{ic}(k-1,n)\rho_{jc}(l-1,n)
%\]
%for $i,j,t \in \{1,\ldots,d\}$, $s, s_1,s_2\in\{1,\ldots,c-1\}$ and $k\neq l \in K$.
%Further define $(Z_{n,1}^{(s)},1\le s<c, Z_{n,k}^{(t)},1\le t\le d, k\in K)$ \Eh{by} to be the standardized
%$(Y_{n,1}^{(s)},1\le s<c, Y_{n,k}^{(t)},1\le t\le d, k\in K)$ so that
%$$Z_{n,1}^{(s)}=\frac{Y_{n,1}^{(s)}-{\rho}_{sc}(0,n)T_n(x_c,z)}{\sqrt{1-{\rho}_{sc}^2(0,n)}},
%\quad
%Z_{n,k}^{(t)}=\frac{Y_{n,k}^{(t)}-\rho_{tc}(k-1,n)T_n(x_c,z)}{\sqrt{1-\rho_{tc}^2(k-1,n)}}$$
Let $\{Y_{n,k,c}^{(i)},1\le i\le d, k\in \{1\}\cup K\}$ have the same distribution as the conditional
distribution of $\{X_{n,k}^{(i)},1\le i\le d, k\in \{1\}\cup K\}$ given $X_{n,1}^{(c)}=T_n(x_c,z)$.
Then
$$\E{Y_{n,k,c}^{(i)}}=\rho_{ic}(k-1,n)T_n(x_c,z)$$
and
$$Cov(Y_{n,k,c}^{(i)},Y_{n,k,c}^{(j)})=\rho_{ij}(|k-l|,n)-\rho_{ic}(k-1,n)\rho_{jc}(l-1,n)$$
for $i,j\in \{1,\ldots,d\}$ and $k,l\in \{1\}\cup K$. Further define
%$\{Z_{n,k,\wzc{c}}^{(i)},1\le i\le d, k\in \{1\}\cup K\}$ \Eh{by} %to be the standardized
%$\{Y_{n,k,c}^{(i)},1\le i\le d, k\in \{1\}\cup K\}$ so that
$$Z_{n,k,c}^{(i)}=\frac{Y_{n,k,c}^{(i)}-\rho_{ic}(k-1,n)T_n(x_c,z)}{\sqrt{1-\rho_{ic}^2(k-1,n)}}, \quad
1\le i\le d, k\in \{1\}\cup K.
$$
Then we have
%\COM{\begin{eqnarray*}
%Cov( Z_{n,1}^{(s_1)}, Z_{n,1}^{(s_2)})
%&=&\frac{{\rho}_{s_1s_2}(0,n)-{\rho}_{s_1c}(0,n){\rho}_{s_2c}(0,n)}
%{\sqrt{(1-{\rho}_{s_1c}^2(0,n))(1-{\rho}_{s_2c}^2(0,n))}} ,\\
%Cov( Z_{n,1}^{(s)}, Z_{n,k}^{(i)})
%&=&\frac{\rho_{is}(k-1,n)-{\rho}_{sc}(0,n)\rho_{ic}(k-1,n)}{\sqrt{(1-{\rho}_{sc}^2(0,n))(1-\rho_{ic}^2(k-1,n))}} ,
%\end{eqnarray*}
%
%\begin{eqnarray*}
%Cov( {Z}_{n,k}^{(i)}, {Z}_{n,k}^{(j)})
%&=&\frac{{\rho}_{ij}(0,n)-\rho_{ic}(k-1,n)\rho_{jc}(k-1,n)}{\sqrt{(1-\rho_{ic}^2(k-1,n))(1-\rho_{jc}^2(k-1,n))}},\\
%Cov( {Z}_{n,k}^{(i)}, {Z}_{n,l}^{(j)})
%&= &\frac{\rho_{ij}(|k-l|,n)-\rho_{ic}(k-1,n)\rho_{jc}(l-1,n)}{\sqrt{(1-\rho_{ic}^2(k-1,n))(1-\rho_{jc}^2(l-1,n))}}
%\end{eqnarray*}
%}
\begin{eqnarray*}
Cov( {Z}_{n,k,c}^{(i)}, {Z}_{n,l,c}^{(j)})
=\frac{\rho_{ij}(|k-l|,n)-\rho_{ic}(k-1,n)\rho_{jc}(l-1,n)}{\sqrt{(1-\rho_{ic}^2(k-1,n))(1-\rho_{jc}^2(l-1,n))}}
\to \frac{\delta_{ic}(k-1)+\delta_{jc}(l-1)-\delta_{ij}(|k-l|)}{2\sqrt{\delta_{ic}(k-1)\delta_{jc}(l-1)}},
\end{eqnarray*}
for $i,j \in \{1,\ldots,d\}$, $k, l \in \{1\} \cup K.$ Thus, using $u_n^2(x)\sim 2 \ln n$ for $x\inr$ we have
\begin{equation}\label{add1}\begin{array}{ll}
&\pk{ Y_{n,1,c}^{(s)}\le u_n(x_s),1\le s<c, Y_{n,k,c}^{(t)}\le u_n(x_t),1\le t\le d, k\in K }\\
=&\mathbb{P}(\frac{1}{2}{\rho}_{sc}(0,n)z+\sqrt{\frac{1+{\rho}_{sc}(0,n)}{2}}\sqrt{\frac{u_n^2(x_c)(1-{\rho}_{sc}(0,n))}{2}}Z_{n,1,c}^{(s)}
\le \frac{1}{2}(u_n(x_s)u_n(x_c)-{\rho}_{sc}(0,n)u_n^2(x_c)),\\
& \quad \frac{1}{2}\rho_{tc}(k-1,n)z+\sqrt{\frac{1+\rho_{tc}(k-1,n)}{2}}\sqrt{\frac{u_n^2(x_c)(1-\rho_{tc}(k-1,n))}{2}}Z_{n,k,c}^{(t)}\\
&\quad \le \frac{1}{2}(u_n(x_t)u_n(x_c)-\rho_{tc}(k-1,n)u_n^2(x_c)),\quad\text{for}\quad 1\le s<c, 1\le t\le d, k \in K )\\
\to& \mathbb{P}(\frac{z}{2}+\sqrt{{\delta}_{sc}(0)}W_{1,c}^{(s)}\le {\delta}_{sc}(0)+\frac{x_s-x_c}{2},1\le s<c, {\delta}_{sc}(0)<\IF,\\
&\quad \frac{z}{2}+\sqrt{\delta_{tc}(k-1)}W_{k,c}^{(t)}\le \delta_{tc}(k-1)+\frac{x_t-x_c}{2},
1\le t\le d,\\
&\quad  \mbox{for all}\ \  k \in K \ \ \mbox{such that}\ \ \delta_{tc}(k-1)<\IF).
\end{array}\end{equation}
%where $\{(W_1^{(s)},1\le s<c, W_{k}^{(t)},1\le t\le d),\widetilde{\delta}_{sc}<\IF, \delta_{tc}(k-1)<\IF,  k\in K\}$
%be jointly normal with zero means and
%\BQNY
%Cov(W_1^{(s_1)}, W_1^{(s_2)})&=&\frac{-\widetilde{\delta}_{s_1s_2}+\widetilde{\delta}_{s_1c}+\widetilde{\delta}_{s_2c}}
%{2\sqrt{\widetilde{\delta}_{s_1c}\widetilde{\delta}_{s_2c}}},\\
%Cov(W_1^{(s)}, W_k^{(i)})&=&\frac{-\delta_{is}(k-1)+\delta_{ic}(k-1)+\widetilde{\delta}_{sc}}{2\sqrt{\widetilde{\delta}_{sc}\delta_{ic}(k-1)}},
%\EQNY
%and
%\BQNY
%Cov(W_k^{(i)}, W_k^{(j)})&=&\frac{-\widetilde{\delta}_{ij}+\delta_{ic}(k-1)+\delta_{jc}(k-1)}{2\sqrt{\delta_{ic}(k-1)\delta_{jc}(k-1)}},\\
%Cov(W_k^{(i)}, W_l^{(j)})&=&\frac{-\delta_{ij}(|k-l|)+\delta_{ic}(k-1)+\delta_{jc}(l-1)}{2\sqrt{\delta_{ic}(k-1)\delta_{jc}(l-1)}}
%\EQNY
%with
%$i,j \in \{1,\ldots,d\}$, $s_1,s_2\in\{1,\ldots,c-1\}$ and $k %\neq l \in K$.
Therefore, \eqref{aeq7} follows by (\ref{add1}) and  \eqref{eq3.4}. The proof of \eqref{aeq8} can be established with similar arguments. Hence the claim follows.
%In order to prove the case $ Similarly, we can get \eqref{aeq8} holds. The proof is complete.
\end{proof}

\begin{lemma}\label{le3.6}
Under the conditions of Theorem \ref{ath1}, for  $c\in \njd$ we have
\begin{equation*}
\lim_{m\to \IF}\limsup_{n\to\IF} \pk{\bigcup_{i=1}^d\bigcup_{j=m}^{r_n}\{X_{n,j}^{(i)}>u_n(x_i)\} \Bigl|X_{n,1}^{(c)}>u_n(x_c)}=0.
\end{equation*}
%{\color{red} (In last lemma, we have $c\ge 2$. Do we need to exclude $c=1$? If not, how to use last lemma for $c=1$?)}
\end{lemma}

\begin{proof} It suffices to show that for each fixed $i\in \njd$
\begin{equation*}
\lim_{m\to \IF}\limsup_{n\to\IF} \pk{\bigcup_{j=m}^{r_n}\{X_{n,j}^{(i)}>u_n(x_i)\}\Bigl|X_{n,1}^{(c)}>u_n(x_c)}=0.
\end{equation*}
As in the proof of Lemma \ref{le3.5}, write with $\EE{a_{nj}(z)=\rho_{ic}(j-1,n)(u_n(x_c)+z/u_n(x_c))} $ and $b_{nj}:=\sqrt{1-\rho_{ic}^2(j-1,n)}$
\begin{eqnarray*}
\pk{\bigcup_{j=m}^{r_n}\{X_{n,j}^{(i)}>u_n(x_i)\}|X_{n,1}^{(c)}>u_n({x_c})} &\sim&\int_0^{\IF}\pk{\bigcup_{j=m}^{r_n}\{a_{nj}(z)+Z_{n,j,c}^{(i)}{b_{nj}}>u_n(x_i)\}}
\EXP{-z-\frac{z^2}{2u_n^2(\tE{x_c})}}\, dz.
\end{eqnarray*}
Hence, we only need to show that for each fixed $z_0>0$
\BQNY
&\lim_{m\to \IF}\limsup_{n\to\IF} \int_{0}^{z_0}
\pk{\bigcup_{j=m}^{r_n}\{a_{nj}(z)+Z_{n,j,c}^{(i)}{b_{nj}}>u_n(x_i)\}}\EXP{-z-\frac{z^2}{2u_n^2({x_c})}}\, dz=0,
\EQNY
 which follows if we show
\begin{equation}\label{add2}
\lim_{m\to\IF}\limsup_{n\to\IF}\sup_{0\le z\le z_0}
\sum_{j=m}^{r_n}\pk{a_{nj}(z)+Z_{n,j,c}^{(i)}{b_{nj}}>u_n(x_i)}=0.
\end{equation}
In view of the derivation of (4.4) in Hsing, H\"usler and Reiss (1996), condition \eqref{eq2.3} implies
\begin{eqnarray*}
\lim_{m\to\IF}\limsup_{n\to\IF}  \max_{m\le j\le r_n}((1-\rho_{ic}(j-1,n))\ln n)^{-1}=0.
\end{eqnarray*}
Thus, for large $n$ and $j\in [m,r_n]$ we have
$${\theta_{nj}:=}\frac{u_n(x_i)-u_n(x_c)\rho_{ic}(j-1,n)}{\sqrt{1-\rho_{ic}^2(j-1,n)}}
-\frac{z\rho_{ic}(j-1,n)}{u_n(x_c)\sqrt{1-\rho_{ic}^2(j-1,n)}}>0.$$
By the fact that
$1-\Phi(x)\le x^{-1}\varphi(x)\quad\text{for} \quad x>0,$
we obtain
\begin{eqnarray*}
\pk{Z_{n,j,c}^{(i)}>{\theta_{nj}}}
%\frac{u_n(x_i)-u_n(x_c)\rho_{ic}(j-1,n)}{\sqrt{1-\rho_{ic}^2(j-1,n)}}
%-\frac{z\rho_{ic}(j-1,n)}{u_n(x_c)\sqrt{1-\rho_{ic}^2(j-1,n)}}}\\
&\le&
%\left(\frac{u_n(x_i)-u_n(x_c)\rho_{ic}(j-1,n)}{\sqrt{1-\rho_{ic}^2(j-1,n)}}
%-\frac{z\rho_{ic}(j-1,n)}{u_n(x_c)\sqrt{1-\rho_{ic}^2(j-1,n)}}\right)^{-1}\\
%&&%\quad \times
\frac1{\theta_{nj}\sqrt{2\pi}}\EXP{-\frac{1}{2}\theta_{nj}^2}.
%\left(\frac{u_n(x_i)-u_n(x_c)\rho_{ic}(j-1,n)}{\sqrt{1-\rho_{ic}^2(j-1,n)}}
%-\frac{z\rho_{ic}(j-1,n)}{u_n(x_c)\sqrt{1-\rho_{ic}^2(j-1,n)}}\right)^2}.
\end{eqnarray*}
Next, for some positive constant $C$  depending only on $x_i,x_c$ and $z_0$ we have
\begin{eqnarray*}
%\left(\frac{u_n(x_i)-u_n(x_c)\rho_{ic}(j-1,n)}{\sqrt{1-\rho_{ic}^2(j-1,n)}}
%-\frac{z\rho_{ic}(j-1,n)}{u_n(x_c)\sqrt{1-\rho_{ic}^2(j-1,n)}}\right)^2
\theta_{nj}^2&\le&C+\frac{1-\rho_{ic}(j-1,n)}{1+\rho_{ic}(j-1,n)}b_n^2\\
&\le&C+\frac{1-\rho_{ic}(j-1,n)}{1+\rho_{ic}(j-1,n)}(2\ln n-\ln\ln n),
\end{eqnarray*}
which implies that
\begin{equation}\label{add3}
\pk{Z_{n,j,c}^{(i)}>\theta_{nj}}
%\frac{u_n(x_i)-u_n(x_c)\rho_{ic}(j-1,n)}{\sqrt{1-\rho_{ic}^2(j-1,n)}}
%-\frac{z\rho_{ic}(j-1,n)}{u_n(x_c)\sqrt{1-\rho_{ic}^2(j-1,n)}}}
\le C^* {b_{nj}^{-1}}n^{-\frac{1-\rho_{ic}(j-1,n)}{1+\rho_{ic}(j-1,n)}}(\ln n)^{-\frac{\rho_{ic}(j-1,n)}{1+\rho_{ic}(j-1,n)}}
%}{\sqrt{1-\rho_{ic}^2(j-1,n)}}
\end{equation}
for some $C^*$ depending on $x_i,x_c$ and $z_0$.
Hence  (\ref{add2}) follows from (\ref{add3}), i.e., the lemma holds. %for some constant $C$ independent of $n,j,i$ and $c$. Thus, by \eqref{eq2.3} we get the desired result.
\end{proof}

\begin{proof}[Proof of Theorem \ref{ath1}] In view of Lemma  \ref{le3.5} %and Lemma \ref{le3.6}}
\begin{equation*}
\lim_{m\to \IF}\lim_{n\to \IF}\pk{\bigcap_{t=1}^{d}\{M_{1,m}^{(t)}\le u_n(x_t)\} \Bigl |X_{n,1}^{(1)}>u_n(x_1)}
=\vartheta_{1}(\tE{\vk{x}})
\end{equation*}
and for $i=2, \cdots, d$
\begin{equation*}
\lim_{m\to \IF}\lim_{n\to \IF}\pk{\bigcap_{s=1}^{i-1}\{M_{m}^{(s)}\le u_n(x_s)\},\bigcap_{t=i}^{d}\{M_{1,m}^{(t)}\le u_n(x_t)\}\Bigl |X_{n,1}^{(i)}>u_n(x_i)}
=\vartheta_{i}(\tE{\vk{x}})\tE{,}
\end{equation*}
with $\vartheta_1(\vk{x})$ and $\vartheta_i(\tE{\vk{x}})$ defined in \eqref{aeq1} and \eqref{aeq6} respectively, 
and by making use of Lemma \ref{le3.6}
\begin{equation*}
\lim_{n\to \IF}\pk{\bigcap_{t=1}^{d}\{M_{1,r_n}^{(t)}\le u_n(x_t)\} \Bigl |X_{n,1}^{(1)}>u_n(x_1)}
=\vartheta_{1}(\tE{\vk{x}})
\end{equation*}
and for $i=2,\cdots,d$
\begin{equation*}
\lim_{n\to \IF}\pk{\bigcap_{s=1}^{i-1}\{M_{r_n}^{(s)}\le u_n(x_s)\},\bigcap_{t=i}^{d}\{M_{1,r_n}^{(t)}\le u_n(x_t)\}\Bigl |X_{n,1}^{(i)}>u_n(x_i)}
=\vartheta_{i}(\tE{\vk{x}}).
\end{equation*}
Hence, by $n(1-\Phi(u_n(x)))\to e^{-x}$ as $n\to \IF$, the theorem follows if further %(set $u_n^{(k)}:=u_n^{(k)}(x_k)$)
\begin{eqnarray*}
&&\pk{\bigcap_{i=1}^{d} \{M_{n}^{(i)}\le u_n(x_i)\}}\\
&&-\exp\left(-n\pk{X_{n,1}^{(1)}>u_n(x_1),\bigcap_{t=1}^d\{M_{1,r_n}^{(t)}\le u_n(x_t)\}}\right.\\
&&\qquad \qquad \qquad\left.-n\sum_{i=2}^d\pk{X_{n,1}^{(i)}>u_n(x_i),\bigcap_{s=1}^{i-1}\{M_{r_n}^{(s)}\le u_n(x_s)\},\bigcap_{t=i}^d\{M_{1,r_n}^{(t)}\le u_n(x_t)\}}\right)\\
&&\to 0\quad\text{as}\quad n\to\IF.
\end{eqnarray*}
Following the arguments  in the proof of Theorem 2.1 in O'Brien (1987),
we first derive  an asymptotic upper bound for $p_{n,d}:=\pk{\bigcap_{i=1}^{d} \{M_{n}^{(i)}\le u_n(x_i)\}}$.
Utilising \eqref{eq3} and Lemma \ref{le1} for all large $n$ we obtain
\begin{eqnarray*}%\label{eq7}
%&&\pk{\bigcap_{k=1}^{d} \{M_{n}^{(k)}\le u_n^{(k)}\}}\nonumber\\
p_{n,d}&=&\left(\pk{\bigcap_{i=1}^{d} \{M_{r_n}^{(i)}\le u_n(x_i)\}}\right)^{q_n}+o(1)\nonumber\\
&=&\left(1-\pk{\bigcup_{i=1}^{d} \{M_{r_n}^{(i)}>u_n(x_i)\}}\right)^{q_n}+o(1)\nonumber\\
&\le&\left(1-r_n\pk{X_{n,1}^{(1)}>u_n(x_1),\bigcap_{t=1}^d\{M_{1,r_n}^{(t)}\le u_n(x_t)\}}\right.\nonumber\\
&&\left.-\sum_{i=2}^dr_n\pk{X_{n,1}^{(i)}>u_n(x_i),\bigcap_{s=1}^{i-1}\{M_{r_n}^{(s)}\le u_n(x_s)\},\bigcap_{t=i}^d\{M_{1,r_n}^{(t)}\le u_n(x_t)\}}\right)^{q_n}+o(1)\nonumber\\
&\le&\exp\left(-n\pk{X_{n,1}^{(1)}>u_n(x_1),\bigcap_{t=1}^d\{M_{1,r_n}^{(t)}\le u_n(x_t)\}}\right.\nonumber\\
&&\qquad \quad \left.-n\sum_{i=2}^d\pk{X_{n,1}^{(i)}>u_n(x_i),\bigcap_{s=1}^{i-1}\{M_{r_n}^{(s)}\le u_n(x_s)\},\bigcap_{t=i}^d\{M_{1,r_n}^{(t)}\le u_n(x_t)\}}\right)+o(1).
\end{eqnarray*}
The rest of the proof is dedicated to the derivation of an asymptotic lower bound for $p_{n,d}$. %Next, we develop a lower bound for $\pk{\bigcap_{k=1}^{d} \{M_{n}^{(k)}\le u_n^{(k)}\}}$.
Choose a sequence of positive integers $\{s_n,n\ge 1 \}$ such that $r_n=o(s_n), \EE{s_n=o(n)}$, and \eqref{eq3} holds with
$r_n$ replaced by $s_n$ and $q_n$ replaced by $t_n=[n/s_n]$. In view of the assumptions  \EE{(see Remark \ref{remE})} this is possible.
% ({\color{red} please say how to do?}).
%For such a sequence we then have  then we have
Since $r_n=o(s_n)$, we have %then  \eqref{aeq1} implies
\begin{equation}\label{eq5}
\pk{M_{r_n}^{(i)}>u_n(x_i)}=o\left(\pk{M_{s_n}^{(i)}> u_n(x_i)}\right), \ \ 1\le i \le d.
\end{equation}
We proceed by induction showing that as $n\to \IF$
\BQN\label{eq6}
&&\lefteqn{\pk{\bigcup_{i=1}^{d} \{M_{s_n}^{(i)}> u_n(x_i)\}}}\\
&=&\left(\sum_{k=1}^{s_n-r_n}\pk{X_{n,k}^{(1)}>u_n(x_1),\bigcap_{t=1}^d\{M_{k,s_n}^{(t)}\le u_n(x_t)\}}\right.\nonumber\\
&&+\left.\sum_{i=2}^d\sum_{k=1}^{s_n-r_n}\pk{X_{n,k}^{(i)}>u_n(x_i),\bigcap_{s=1}^{i-1}\{M_{k-1,s_n}^{(s)}\le u_n(x_s)\},\bigcap_{t=i}^d\{M_{k,s_n}^{(t)}\le u_n(x_t)\}}\right)(1+o(1)).\nonumber
\EQN
If  $d=1$, \cw{as in O'Brien (1987)}, we have
\begin{eqnarray*}
\pk{M_{s_n}^{(1)}> u_n(x_1)}
&=&\pk{M_{s_n-r_n}^{(1)}>u_n(x_1),M_{s_n-r_n,s_n}^{(1)}\le u_n(x_1)}+\pk{M_{r_n}^{(1)}> u_n(x_1)}\\
%&=&\pk{M_{s_n-r_n}^{(1)}>u_n^{(1)},M_{s_n-r_n,s_n}^{(1)}\le u_n^{(1)}}(1+o(1))\\
&=&\left(\sum_{k=1}^{s_n-r_n}\pk{X_{n,k}^{(1)}>u_n(x_1), M_{k,s_n}^{(1)}\le u_n(x_1)}\right)(1+o(1)) \quad \text{as}\quad  n\to \IF.
\end{eqnarray*}
For $d=2$, by \eqref{addeq7}, \eqref{aeq2} and stationarity we have
\begin{eqnarray*}
&&\pk{\{M_{s_n}^{(1)}>u_n(x_1)\}\cup \{M_{s_n}^{(2)}>u_n(x_2)\}}\\
&=&\pk{M_{s_n}^{(1)}>u_n(x_1)}+\pk{M_{s_n}^{(2)}>u_n(x_2),M_{s_n}^{(1)}\le u_n(x_1)}\\
&=&\pk{M_{s_n}^{(1)}>u_n(x_1)}+\pk{M_{s_n}^{(2)}>u_n(x_2),M_{s_n}^{(1)}\le u_n(x_1)}\\
&&-\pk{M_{s_n-r_n,s_n}^{(2)}>u_n(x_2),M_{s_n-r_n,s_n}^{(1)}\le u_n(x_1)}+\pk{M_{r_n}^{(2)}>u_n(x_2),M_{r_n}^{(1)}\le u_n(x_1)}\\
&=&\left(\sum_{k=1}^{s_n-r_n}\pk{X_{n,k}^{(1)}>u_n(x_1),M_{k,s_n}^{(1)}\le u_n(x_1)}\right.
+\sum_{k=1}^{s_n}\pk{X_{n,k}^{(2)}>u_n(x_2), M_{k,s_n}^{(2)}\le u_n(x_2),M_{k-1,s_n}^{(1)}\le u_n(x_1)}\\
&&-\sum_{k=1}^{s_n-1}\pk{M_{k,s_n}^{(2)}>u_n(x_2),X_{n,k}^{(1)}>u_n(x_1),M_{k,s_n}^{(1)}\le u_n(x_1)}\\
&&-\sum_{k=s_n-r_n+1}^{s_n}\pk{X_{n,k}^{(2)}>u_n(x_2), M_{k,s_n}^{(2)}\le u_n(x_2),M_{k-1,s_n}^{(1)}\le u_n(x_1)}\\
&&\left.+\sum_{k=s_n-r_n+1}^{s_n-1}\pk{M_{k,s_n}^{(2)}>u_n(x_2),X_{n,k}^{(1)}>u_n(x_1),M_{k,s_n}^{(1)}\le u_n(x_1)}
\right)(1+o(1))\\
&=&\left(\sum_{k=1}^{s_n-r_n}\pk{X_{n,k}^{(1)}>u_n(x_1),M_{k,s_n}^{(1)}\le u_n(x_1),M_{k,s_n}^{(2)}\le u_n(x_2)}\right.\\
&&\left.+\sum_{k=1}^{s_n-r_n}\pk{X_{n,k}^{(2)}>u_n(x_2), M_{k,s_n}^{(2)}\le u_n(x_2),M_{k-1,s_n}^{(1)}\le u_n(x_1)}\right)(1+o(1)),
\end{eqnarray*}
i.e., \eqref{eq6} holds for $d=2$.
Assume next that  \eqref{eq6} holds for $d=m-1>2$.  By \eqref{eq5}
\begin{eqnarray*}
\pk{\bigcup_{i=1}^m \{M_{s_n}^{(i)}>u_n(x_i)\}}&=&\pk{\bigcup_{i=1}^{m-1} \{M_{s_n}^{(i)}> u_n(x_i)\}}+\pk{\bigcap_{i=1}^{m-1} \{M_{s_n}^{(i)}\le u_n(x_i)\}, M_{s_n}^{(m)}>u_n(x_m)}\\
&=&\Biggl( \pk{\bigcup_{i=1}^{m-1} \{M_{s_n}^{(i)}> u_n(x_i)\}} +
\pk{\bigcap_{i=1}^{m-1} \{M_{s_n}^{(i)}\le u_n(x_i)\}, M_{s_n}^{(m)}>u_n(x_m)}\\
&&-\pk{\bigcap_{i=1}^{m-1} \{M_{s_n-r_n,s_n}^{(i)}\le u_n(x_i)\}, M_{s_n-r_n,s_n}^{(m)}>u_n(x_m)}\Biggr)(1+o(1)).
\end{eqnarray*}
Consequently  \eqref{eq2} implies that \eqref{eq6} holds for $d=m$. According to \eqref{eq6}, by stationarity
we have
\begin{eqnarray*}
&&\pk{\bigcup_{i=1}^d \{M_{s_n}^{(i)}>u_n(x_i)\}}\\
&\le&\left(\sum_{k=1}^{s_n-r_n}\pk{X_{n,k}^{(1)}>u_n(x_1),\bigcap_{t=1}^d\{M_{k,r_n+k-1}^{(t)}\le u_n(x_t)\}}\right.\\
&&\left.+\sum_{i=2}^d\sum_{k=1}^{s_n-r_n}\pk{X_{n,k}^{(i)}>u_n(x_i),
\bigcap_{s=1}^{i-1}\{M_{k-1,r_n+k-1}^{(s)}\le u_n(x_s)\},\bigcap_{t=i}^d\{M_{k,r_n+k-1}^{(t)}\le u_n(x_t)\}}\right)(1+o(1))\\
&\le&\left(s_n\pk{X_{n,1}^{(1)}>u_n(x_1),
\bigcap_{t=1}^d\{M_{1,r_n}^{(t)}\le u_n(x_t)\}}\right.\\
&&\left.+\sum_{i=2}^ds_n\pk{X_{n,1}^{(i)}>u_n(x_i),
\bigcap_{s=1}^{i-1}\{M_{r_n}^{(s)}\le u_n(x_s)\},\bigcap_{t=i}^d\{M_{1,r_n}^{(t)}\le u_n(x_t)\}}\right)(1+o(1)).
\end{eqnarray*}
Since by our choice of the sequence $\{s_n,n\ge 1\}$
\begin{equation*}\label{eq4}
%\pk{\bigcap_{k=1}^{d} \{M_{n}^{(k)}\le u_n^{(k)}\}}
p_{n,d}=\left(\pk{\bigcap_{i=1}^{d} \{M_{s_n}^{(i)}\le u_n(x_i)\}}\right)^{t_n}+o(1) \quad\text{as}\quad n\to \IF
\end{equation*}
we have
\begin{eqnarray*}%\label{eq8}
&&\pk{\bigcap_{i=1}^{d} \{M_{n}^{(i)}\le u_n(x_i)\}}\\
&\ge&\exp\left(-n\pk{X_{n,1}^{(1)}>u_n(x_1),\bigcap_{t=1}^d\{M_{1,r_n}^{(t)}\le u_n(x_t)\}}\right.\\
&&\qquad \qquad\left.-n\sum_{i=2}^d\pk{X_{n,1}^{(i)}>u_n(x_i),\bigcap_{s=1}^{i-1}\{M_{r_n}^{(s)}\le u_n(x_s)\},\bigcap_{t=i}^d\{M_{1,r_n}^{(t)}\le u_n(x_t)\}}\right)+o(1).
\end{eqnarray*}
 Hence the theorem holds. \end{proof}

\noindent\textbf{Acknowledgments.}
Research of Hashorva and  Weng  was   supported  by the Swiss
National Science Foundation grants 200021-134785, 200021-140633/1 and RARE -318984 (an FP7 Marie Curie IRSES Fellowship).

\section*{References}
\vskip .001in
\begin{itemize}
\itemindent -.35in
\itemsep-.05in

%\item[]  S.M. Berman (1982). Sojourns and extremes of stationary processes. {\it Ann. Probab.} {\bf 10}, 1--46.

\item[]  J.P. French and R.A. Davis (2013). The asymptotic distribution of the maxima of a Gaussian random field on a lattice. {\it Extremes}, {\bf 16}, 1--26.

%\item[]  B. Das, S. Engelke and  Hashorva, E. (2014). Extremal behavior of squared Bessel processes attracted by the Brown-Resnick process. Submitted.

\item[] S. Engelke, Z. Kabluchko and M. Schlather (2014).
Maxima of independent, non-identically distributed Gaussian vectors. {\it Bernoulli}, in press.

%\item[] S. Engelke, A. Malinowski and M. Schlather (2014). Estimation of H\"{u}sler-Reiss distributions and Brown-Resnick processes. Available from http://arxiv.org/abs/1207.6886.

\item[] B.G. Manjunath, M. Frick and R.-D. Reiss (2013).  Some notes on extremal discriminant analysis. {\it J. Multiv. Analysis} {\bf 103}, 107--115.

\item[] M. Frick and R.-D. Reiss (2013). Expansions and penultimate distributions of
maxima of bivariate normal random vectors. {\it Statist. \& Probab. Lett.} {\bf 83}, 2563--2568.

%\item[] E. Hashorva (2008). Extremes of weighted Dirichlet arrays. {\it Extremes} {\bf 11}, 393--420.

\item[] E. Hashorva (2013).  Minima and maxima of elliptical triangular arrays and spherical processes.
{\it Bernoulli} {\bf 19}, 886--904.

\item[] E. Hashorva, Z. Kabluchko and A. W\"ubker (2012).  Extremes of independent chi-square random vectors. {\it Extremes} {\bf 15}, 35--42.

\item[] E. Hashorva and Z. Weng (2013). Limit laws for extremes of dependent stationary Gaussian arrays. {\it Statist. \& Probab. Lett.}
{\bf 83}, 320--330.

\item[] T. Hsing, J. H\"usler and R.-D. Reiss (1996). The extremes of a triangular array of normal random variables. {\it Ann. Appl. Probab.} {\bf 6}, 671--686.

\item[] J. H\"usler and R.-D. Reiss (1989). Maxima of normal random vectors: between independence and complete dependence. {\it Statist. \& Probab. Lett.} {\bf 7}, 283--286.

\item[] W.V. Li and Q.M. Shao (2001). Gaussian processes: inequalities, small ball probabilities and applications. {\it Stochastic Processes: Theory and Applications} {\bf 19}, 533--597.

\item[] A.J. McNeil, R. Frey and P. Embrechts (2005).  Quantitative Risk Management: Concepts, Techniques, and Tools.
 {\it Princeton University Press}.

\item[] G.L. O'Brien. (1987). Extreme values for stationary and Markov sequences. {\it Ann. Probab.} {\bf 15}, 281--291.

\item[] V.I. Piterbarg (1996). Asymptotic Methods in the Theory of Gaussian Processes and Fields.
{\it American Mathematical Society, Providence, RI}.

\end{itemize}
\end{document}